\documentclass{article}
\usepackage{amsmath}
\usepackage{amssymb}
\usepackage{amsthm,bbm}

\newcommand{\bunderline}[1]{\underline{#1\mkern-4mu}\mkern4mu }

\newcommand{\one}{\mathbbm 1}

\def\reals{\mathbb{R}}

\def\ereals{\overline{\mathbb{R}}}

\def\rinterior{\mathop{\rm rint}}

\def\comp{\raise 1pt \hbox{$\scriptstyle\circ$}}

\def\argmin{\mathop{\rm argmin}\limits}

\def\minimize{\mathop{\rm minimize}\limits}

\def\st{\mathop{\rm subject\ to}}

\def\dom{\mathop{\rm dom}\nolimits}

\def\upto{{\raise 1pt \hbox{$\scriptstyle \,\nearrow\,$}}}
\def\downto{{\raise 1pt \hbox{$\scriptstyle \,\searrow\,$}}}

\def\acl{\mathop{\rm acl}\nolimits}
\def\cl{\mathop{\rm cl}\nolimits}

\def\tos{\rightrightarrows}
\def\FF{(\F_t)_{t=0}^T}
\def\one{\mathbbm 1}
\def\ovr{\mathop{\rm over}}

\def\B{{\cal B}}

\def\F{{\cal F}}

\def\N{{\cal N}}

\def\R{{\mathbb R}}

\def\Q{{\cal Q}}

\def\U{{\cal U}}

\def\Y{{\cal Y}}

\newtheorem{theorem}{Theorem}
\newtheorem{lemma}[theorem]{Lemma}

\newtheorem{example}[theorem]{Example}

\theoremstyle{definition}

\title{Duality and optimality conditions in stochastic optimization and mathematical finance}
\author{Sara Biagini\footnote{University of Pisa} \and Teemu Pennanen\footnote{King's College London} \and Ari-Pekka Perkki\"o\footnote{Technische Universit\"at Berlin. The author is grateful to the Einstein Foundation for the financial support.}}

\begin{document}

\maketitle

\begin{abstract}
This article studies convex duality in stochastic optimization over finite discrete-time. The first part of the paper gives general conditions that yield explicit expressions for the dual objective in many applications in operations research and mathematical finance. The second part derives optimality conditions by combining general saddle-point conditions from convex duality with the dual representations obtained in the first part of the paper. Several applications to stochastic optimization and mathematical finance are given.
\end{abstract}

\section{Introduction}

Let $(\Omega,\F,P)$ be a complete probability space with a filtration $\FF$ of complete sub $\sigma$-algebras of $\F$ and consider the dynamic stochastic optimization problem
\begin{equation}\label{p}\tag{$P_u$}
\minimize\quad Ef(x,u):=\int f(x(\omega),u(\omega),\omega)dP(\omega)\quad\text{over $x\in\N$}
\end{equation}
parameterized by a measurable function $u\in L^0(\Omega,\F,P;\reals^m)$. Here and in what follows,
\[
\N := \{(x_t)_{t=0}^T\,|\,x_t\in L^0(\Omega,\F_t,P;\reals^{n_t})\},
\]
for given integers $n_t$ and $f$ is an extended real-valued $\B(\reals^n\times\reals^m)\otimes\F$-measurable function on $\reals^n\times\reals^m\times\Omega$, where $n:=n_0+ \ldots +n_T$. The variable $x\in\N$ is interpreted as a decision strategy where $x_t$ is the decision taken at time $t$. Throughout this paper, we define the expectation of a measurable function $\phi$ as $+\infty$ unless the positive part $\phi^+$ is integrable (In particular, the sum of extended real numbers is defined as $+\infty$ if any of the terms equals $+\infty$). The function $Ef$ is thus well-defined extended real-valued function on $\N\times L^0(\Omega,\F,P;\reals^m)$. We will assume throughout that the function $f(\cdot,\cdot,\omega)$ is {\em proper}, {\em lower semicontinuous} and {\em convex} for every $\omega\in\Omega$.

It was shown in \cite{pen11c} that, when applied to \eqref{p}, the conjugate duality framework of Rockafellar~\cite{roc74} allows for a unified treatment of many well known duality frameworks in stochastic optimization and mathematical finance. An important step in the analysis is to derive dual expressions for the optimal value function
\[
\varphi(u) := \inf_{x\in\N}Ef(x,u)
\]
over an appropriate subspace of $L^0$. In this context, the absence of a duality gap is equivalent to the closedness of the value function. Pennanen and Perkki\"o~\cite{pp12} and more recently Perkki\"o~\cite{per14b} gave conditions that guarantee that $\varphi$ is closed and that the optimum in \eqref{p} is attained for every $u\in L^0$. The given conditions provide far reaching generalizations of well-known no-arbitrage conditions used in financial mathematics. 

The present paper makes two contributions to the duality theory for \eqref{p}. First, we extend the general duality framework of \cite{pen11c} by allowing more general dualizing parameters and by relaxing the time-separability property of the Lagrangian. We show that, under suitable conditions, the expression in \cite[Theorem~2.2]{pen11c} is still valid in this extended setting. This also provides a correction to \cite[Theorem~2.2]{pen11c} which omitted certain integrability conditions that are needed in general; see \cite{pen15}. Second, we give optimality conditions for the optimal solutions of \eqref{p}. Again, we follow the general conjugate duality framework of \cite{roc74} by specializing the saddle-point conditions to the present setting. The main difficulty here is that, in general, the space $\N$ does not have a proper topological dual so we cannot write the generalized Karush-Kuhn-Tucker condition in terms of subgradients. Nevertheless, the dual representations obtained in the first part of the paper allow us to write the saddle-point conditions more explicitly in many interesting applications.

In the case of perfectly liquid financial markets, we recover well-known optimality conditions in terms of martingale measures. For Kabanov's currency market model with transaction costs~\cite{kab99}, we obtain optimality conditions in terms of dual variables that extend the notion of a ``consistent price system'' to possibly nonconical market models. We treat problems of convex optimal control under the generalized framework of Bolza much as in \cite{rw83}. Our formulation and its embedding in the conjugate duality framework of \cite{roc74} is slightly different from that in \cite{rw83}, however, so direct comparisons are not possible. Our formulation is motivated by applications in mathematical finance. In particular, the optimality conditions for the currency market model are derived by specializing those obtained for the problem of Bolza.

\section{Duality}\label{sec:duality}

From now on, we will assume that the parameter $u$ belongs to a decomposable space $\U\subset L^0$ which is in separating duality with another decomposable space $\Y\subset L^0$ under the bilinear form
\[
\langle u,y\rangle = E(u\cdot y).
\]
Recall that $\U$ is {\em decomposable} if
\[
\one_Au+\one_{\Omega\setminus A}u'\in\U
\]
whenever $A\in\F$, $u\in\U$ and $u'\in L^\infty$; see e.g.\ \cite{roc76}. Examples of such dual pairs include the Lebesgue spaces $\U=L^p$ and $\Y=L^q$ and decomposable pairs of Orlicz spaces; see Section~\ref{sec:oc}. The {\em conjugate} of $\varphi:\U\to\ereals$ is the extended real-valued convex function on $\Y$ defined by
\[
\varphi^*(y)=\sup_{u\in \U}\{\langle u,y\rangle-\varphi(u)\}.
\]
Here and in what follows, $\ereals:=\reals\cup\{+\infty,-\infty\}$. If $\varphi$ is closed\footnote{A convex function is {\em closed} if it is lower semicontinuous and either proper or a constant. A function is {\em proper} if it never takes the value $-\infty$ and it is finite at some point.} with respect to the weak topology induced on $\U$ by $\Y$, the biconjugate theorem (see e.g.\ \cite[Theorem~5]{roc74}) gives the {\em dual representation}
\begin{align*}
\varphi(u) = \sup_{y\in\Y}\{\langle u,y\rangle - \varphi^*(y)\}.
\end{align*}
This simple identity is behind many well known duality relations in operations research and mathematical finance. It was shown in \cite{pp12} that appropriate generalizations of the {\em no-arbitrage} condition from mathematical finance guarantee the closedness of $\varphi$. Recently, the conditions were extended in \cite{per14b} to allow for more general objectives.

In general, it may be difficult to derive more explicit expressions for $\varphi^*$. Following \cite{roc74}, one can always write the conjugate as
\[
\varphi^*(y) =-\inf_{x\in\N}L(x,y),
\]
where the {\em Lagrangian} $L:\N\times\Y\to\ereals$ is defined by
\begin{equation*}
L(x,y)=\inf_{u\in\U}\{Ef(x,u)-\langle u,y\rangle \}.
\end{equation*}
We will show that, under appropriate conditions, the second infimum in 
\[
\varphi^*(y) =-\inf_{x\in\N}\inf_{u\in\U}\{Ef(x,u)-\langle u,y\rangle \}
\]
can be taken scenariowise while the first infimum may be restricted to the space $\N^\infty$ of essentially bounded strategies. Both infima are easily calculated in many interesting applications; see \cite{pen11c} and the examples below.
 
Accordingly, we define the {\em Lagrangian integrand} on $\reals^n\times\reals^m\times\Omega$ by
\[
l(x,y,\omega) := \inf_{u\in\R^m}\{f(x,u,\omega)-u\cdot y\}.
\]
We will also need the pointwise conjugate of $f$:
\begin{align*}
f^*(v,y,\omega) :&= \sup_{x\in\reals^{n},u\in\reals^m}\{x\cdot v+u\cdot y-f(x,u,\omega)\}\\
&= \sup_{x\in\reals^{n}}\{x\cdot v-l(x,y,\omega)\}.
\end{align*}
By \cite[Theorem~14.50]{rw98}, the pointwise conjugate of a normal integrand is also a normal integrand. Clearly, $l$ is upper semicontinuous in the second argument since it is the pointwise infimum of continuous functions of $y$. Similarly, 
\[
\underline l(x,y,\omega) := \sup_{v\in\R^n}\{x\cdot v-f^*(v,y,\omega)\}
\]
is lower semicontinuous in the first argument. In fact, $\underline l(\cdot,y,\omega)$ is the biconjugate of $l(\cdot,y,\omega)$ while $-l(x,\cdot,\omega)$ is the biconjugate of $-\underline l(x,\cdot,\omega)$; see \cite[Theorem~34.2]{roc70a}. The function $(x,\omega)\mapsto \underline l(x,y(\omega),\omega)$ is a normal integrand for any $y\in\Y$ while $(y,\omega)\mapsto - l(x(\omega),y,\omega)$ is a normal integrand for any $x\in\N$ and thus, the integral functionals $E\underline l$ and $El$ are well-defined on $\N\times\Y$ (recall our convention of defining an integral as $+\infty$ unless the positive part of the integrand is integrable). Indeed, we have $l(x(\omega),y,\omega)=-h^*(y,\omega)$ for $h(u,\omega):=f(x(\omega),u,\omega)$, where  $h$ is a normal integrand, by \cite[Proposition~14.45(c)]{rw98}. Similarly for $\underline l$.

Restricting strategies to the space $\N^\infty\subset\N$ of essentially bounded strategies gives rise to the auxiliary value function
\[
\tilde\varphi(u) = \inf_{x\in\N^\infty}Ef(x,u).
\]
Under the conditions of Theorem~\ref{thm:duality} below, the conjugates of $\tilde\varphi$ and $\varphi$ coincide, or in other words, closures of $\tilde\varphi$ and $\varphi$ are equal. The following lemma from \cite{per14b} will play an important role. We denote
\[
\N^\perp := \{v\in L^1(\Omega,\F,\reals^n)\,|\,E(x\cdot v)=0\ \forall x\in\N^\infty\}.
\]
\begin{lemma}\label{lem:1}
Let $x\in\N$ and $v\in\N^\perp$. If $E[x\cdot v]^+\in L^1$, then $E(x\cdot v)=0$.
\end{lemma}

We will use the notation
\begin{align*}
\dom_1 Ef &:=\{x\in\N\,|\, \exists u\in\U:\ Ef(x,u)<+\infty\}.
\end{align*}
Recall that {\em algebraic closure}, $\acl C$, of a set $C$ is the set of points $x$ such that $(x,z]\subset C$ for some $z\in C$. Clearly, $C\subseteq\acl C$ while in a topological vector space, $\acl C\subseteq\cl C$.

\begin{theorem}\label{thm:duality}
If $\dom El(\cdot,y)\cap \N^\infty\subseteq\acl(\dom_1Ef\cap\N^\infty)$, then
\[
\tilde\varphi^*(y) = -\inf_{x\in\N^\infty}El(x,y).
\]
If, for every $x\in\N^\infty$ with $x\in\cl\dom_1f$ almost surely, there exists $\bar x\in\N^\infty$ with $\bar x\in\rinterior\dom_1f$ almost surely and $(\bar x,x)\in\dom_1 Ef$, then
\[
\tilde\varphi^*(y) = -\inf_{x\in\N^\infty}E\underline l(x,y).
\]

We always have
\[
\tilde\varphi^*(y)\le\varphi^*(y)\le \inf_{v\in\N^\perp} Ef^*(v,y).
\]
In particular, if there exists $v\in\N^\perp$ such that $\tilde\varphi^*(y)=Ef^*(v,y)$, then
\[
\tilde\varphi^*(y)=\varphi^*(y).
\]
\end{theorem}

\begin{proof}
By the interchange rule (\cite[Theorem~14.60]{rw98}), $L(x,y) = El(x,y)$ for $x\in\dom_1Ef$ and thus,
\begin{align*}
-\tilde\varphi^*(y) &= \inf_{x\in\N^\infty}L(x,y) = \inf_{x\in\N^\infty\cap\dom_1Ef}L(x,y)\\
&= \inf_{x\in\N^\infty\cap\dom_1Ef}El(x,y)\ge \inf_{x\in\N^\infty}El(x,y).
\end{align*} 
The converse holds trivially if $\dom El(\cdot,y)\cap\N^\infty\ne \emptyset$. Otherwise, let $a\in\reals$ and $\tilde x\in\N^\infty$ be such that $El(\tilde x,y)< a$.  By the first assumption, there exists an $x'\in \dom_1Ef\cap\N^\infty$ such that $(1-\lambda)\tilde x+\lambda x'\in\dom_1Ef$ for all $\lambda\in(0,1]$. By convexity, $L((1-\lambda)\tilde x+\lambda x',y)=El((1-\lambda)\tilde x+\lambda x',y)<a$ for $\lambda$ small enough and thus
\[
-\tilde\varphi^*(y)=\inf_{x\in\N^\infty} El(x,y).
\] 

To prove the second claim, let $y\in\Y$ and $x\in\dom El(\cdot,y)\cap\N^\infty$. By \cite[Theorem 34.3]{roc70a}, $x\in\cl\dom_1 f$ almost surely so, by assumption, there exists $\bar x\in\N^\infty$ with $\bar x\in\rinterior\dom_1f$ almost surely and $(x,\bar x)\in\dom_1 Ef$. Thus, $x\in\acl\dom_1 Ef$ so the first domain condition is satisfied. By \cite[Theorem 6.1]{roc70a}, $(x,\bar x)\subset\rinterior\dom_1 f$, and thus by the same line segment argument as above, the infimum in 
\[
-\tilde\varphi^*(y) = \inf_{x\in\N^\infty} El(x,y) 
\]
can be restricted to those $x$ for which $x\in\rinterior\dom_1 f$. Then, by \cite[Theorem 34.2]{roc70a}, we may replace $l$ by $\underline l$ without affecting the infimum.

As to the last claim, the Fenchel inequality gives 
\begin{equation*}
f(x,u)+f^*(v,y)\ge x\cdot v+u\cdot y.
\end{equation*}
Therefore, for $(x,u)\in\dom Ef$ and $v\in\N^\perp$ with $Ef^*(v,y)<\infty$, we have $E(x\cdot v)=0$ by Lemma~\ref{lem:1}, so we get
\[
\varphi^*(y)=\sup_{x\in\N,u\in\U}E[u\cdot y-f(x,u)]\le \inf_{v\in\N^\perp}Ef^*(v,y).
\]
Since $\tilde\varphi\ge\varphi$, have $\tilde\varphi^*\le \varphi^*$. This completes the proof of the inequalities. The last statement concerning the equality clearly follows from the inequalities. 

\end{proof}

Theorem~\ref{thm:duality} gives conditions under which the conjugate of the value function of \eqref{p} can be expressed as
\[
\varphi^*(y) = \inf_{x\in\N^\infty}El(x,y).
\]
The second part gives conditions that allow one to replace $l$ by $\underline l$ in the above expression. One can then resort to the theory of normal integrands when calculating the infimum. It was shown in \cite{pen11c} that this yields many well-known dual expressions in operations research and mathematical finance. Unfortunately, the proof of the above expression in \cite[Theorem~2.2]{pen11c} was incorrect and the given conditions are not sufficient in general; see \cite{pen15}. The following example illustrates what can go wrong if the condition on the domains is omitted. Here and in what follows, $\delta_C$ denotes the indicator function of a set $C$: $\delta_C(x)$ equals $0$ or $+\infty$ depending on whether $x\in C$ or not.

\begin{example}\label{ex:ie0}
Let $\F_0$ be trivial, $n_0=1$, $\U=L^\infty$, $\Y=L^1$, $\beta\in L^2$ be such that $\beta^+$ and $\beta^-$ are unbounded, and let
\[
f(x,u,\omega)= \delta_{\reals_-}(\beta(\omega)x_0+u)
\]
so that $\dom Ef=\{x\in\N\mid x_0 = 0\}\times L^\infty_-$, $\tilde\varphi^*=\delta_{L^1_+}$ and
\[
l(x,y,\omega)=y\beta(\omega)x_0-\delta_{\reals_+}(y).
\]
For $y=\beta^+$, we get $\inf_{x\in\N^\infty}El(x,y)=-\infty$ while $\tilde\varphi^*(y)=0$. Here $\dom El(\cdot,y)=\N$, so the first condition in Theorem~\ref{thm:duality} is violated.
\end{example}

The following example shows how the equality $\varphi^*=\tilde\varphi^*$ may fail to hold even when the first condition of Theorem~\ref{thm:duality} is satisfied.

\begin{example}
Let $T=1$, $n=2$, $\F_0$ is the completed trivial $\sigma$-algebra, and
\[
f(x,u,\omega)=|x_0-1|+\delta_{\reals_-}(\alpha(\omega) |x_0|-x_1) + \frac{1}{2}|u|^2,
\]
where $\alpha\in L^0(\F_1)$ is positive and unbounded. It is easily checked that with $\U=\Y=L^2$ the first condition of Theorem~\ref{thm:duality} holds but $\tilde\varphi(u)=1+\frac{1}{2}\|u\|^2$ and $\varphi(u)=\frac{1}{2}\|u\|^2$.
\end{example}

At the moment, it is an open question whether the algebraic closure in the domain condition could be replaced by a topological closure. This can be done, however, e.g.\ if $El(\cdot,y)$ is upper semicontinuous on the closure of $\dom_1Ef\cap\N^\infty$. Also, in the deterministic case (where $P$ is integer-valued on $\F$), the first condition is automatically satisfied since $\dom l(\cdot,y)\subset\cl\dom_1f$ for all $y$, by \cite[Theorem~34.3]{roc70a}.

The following describes a general situation where the assumptions of Theorem~\ref{thm:duality} are satisfied.

\begin{example}\label{ex:ie}
Consider the problem
\begin{alignat*}{2}
&\minimize\quad& E&f_0(x)\quad\text{{\rm over} $x\in\N$}\\
&\st\quad& &f_j(x)\le 0\ P\text{-a.s.},\ j=1,\ldots,m,
\end{alignat*}
where $f_j$ are normal integrands. The problem fits the general framework with 
\begin{align*}
f(x,u,\omega) &=
\begin{cases}
f_0(x,\omega) & \text{if $f_j(x,\omega)+u_j\le 0$ for $j=1,\ldots,m$},\\
+\infty & \text{otherwise}.
\end{cases}
\end{align*}
This model was studied in Rockafellar and Wets~\cite{rw78} who gave optimality conditions in terms of dual variables. We will return to optimality conditions in the next section. For now, we note that if $f_j(x)\in\U$ for all $x\in\N^\infty$ and $j=0,\ldots,m$, then the assumptions of Theorem~\ref{thm:duality} are satisfied. 

Indeed, $\dom_1 Ef\cap\N^\infty=\N^\infty$ so the first two conditions in Theorem~\ref{thm:duality} hold. The lagrangian integrand can now be written as
\[
l(x,y,\omega) = f_0(x,\omega) + \sum_{j=1}^my_jf_j(x,\omega).
\]
By \cite[Theorem~22]{roc74}, the function $El(\cdot,y)$ is Mackey-continuous (with respect to the usual pairing of $L^\infty$ and $L^1$) at the origin of $L^\infty$ for every $y\in\Y$ so, by convexity, the function
\[
\phi_y(\tilde x) := \inf_{x\in\N^\infty}El(x+\tilde x,y),
\]
is Mackey-continuous as well; see e.g.~\cite[Theorem~8]{roc74}. By \cite[Theorem~11]{roc74}, this implies that $\phi_y$ has a subgradient at the origin, i.e.\ a $v\in L^1$ such that
\[
El(x+\tilde x,y) \ge \phi_y(0) + E(\tilde x\cdot v)\qquad \forall \tilde x\in L^\infty,\ x\in \N^\infty
\]
or equivalently 
\[
El(\tilde x,y) \ge \phi_y(0) + E((\tilde x-x)\cdot v)\qquad \forall \tilde x\in L^\infty,\ x\in \N^\infty.
\]
This implies $v\in \N^\perp$ and $\inf_{\tilde x\in L^\infty} E[l(\tilde x,y)-\tilde x\cdot v]=\phi_y(0)$. By the interchange rule (\cite[Theorem~14.60]{rw98}), this can be written as $-Ef^*(v,y)=\phi_y(0)=-\tilde\varphi^*(y)$, so the last condition of Theorem~\ref{thm:duality} holds.
\end{example}

Section~4 of \cite{pen11} studied financial models for pricing and hedging of portfolio-valued contingent claims (claims with physical delivery) along the lines of Kabanov~\cite{kab99}; see Example~\ref{ex:kabanov} below. The following example is concerned with the more classical financial model with claims with cash-delivery.

\begin{example}\label{ex:optinv}
Consider the problem
\begin{equation}\label{alm}\tag{ALM}
\minimize\quad E  V\left(u-\sum_{t=0}^{T-1} x_t\cdot\Delta s_{t+1}\right)\quad\ovr\quad x\in\N,
\end{equation}
where $V$ convex normal integrand on $\reals\times\Omega$ such that $V(\cdot,\omega)$ is nondecreasing nonconstant and $V(0,\omega)=0$. This models the optimal investment problem of an agent with financial liabilities $u\in\U$ and a ``disutility function'' $V$. The $\F_t$-measurable vector $s_t$ gives the unit prices of ``risky'' assets at time $t$ and the vector $x_t$ the units held over $(t,t+1]$; see e.g. R{\'a}sonyi and Stettner~\cite{rs5} and the references therein.

Assume that, for every $x\in\N^\infty$, there is a $u\in\U$ such that $EV(u-\sum_{t=0}^{T-1} x_t\cdot\Delta s_{t+1})<\infty$, then the closure of the value function $\varphi$ of \eqref{alm} has the dual representation
\begin{align*}
(\cl\varphi)(u)=\sup_{y\in\Q}E[uy - V^*(y)],
\end{align*}
where $\Q$ is the set of positive multiples of martingale densities $y\in\Y$, i.e.\ densities $dQ/dP$ of probability measures $Q\ll P$ under which the price process $s$ is a martingale.

Indeed, \eqref{alm} fits the general model \eqref{p} with
\begin{align*}
f(x,u,\omega) &= V\left(u-\sum_{t=0}^{T-1} x_t\cdot\Delta s_{t+1}(\omega),\omega\right),\\
l(x,y,\omega) &=-V^*(y,\omega)-y\sum_{t=0}^{T-1}x_t\cdot\Delta s_{t+1}(\omega),\\
f^*(v,y,\omega) &=
\begin{cases}
V^*(y,\omega) & \text{if $v_t=-y\Delta s_{t+1}(\omega)$ for $t<T$ and $v_T=0$},\\
+\infty & \text{otherwise}.
\end{cases}
\end{align*}
The integrability condition implies $\dom_1 Ef\cap \N^\infty=\N^\infty$ so the first two conditions of Theorem~\ref{thm:duality} hold. Thus,
\begin{align*}
-\tilde\varphi^*(y)=\inf_{x\in\N^\infty}El(x,y) = -EV^*(y) + \inf_{x\in\N^\infty}E\left[-\sum_{t=0}^{T-1}x_t\cdot(y\Delta s_{t+1})\right].
\end{align*}
By Fenchel inequality,
\[
uy-\sum_{t=0}^{T-1}x_t\cdot(y\Delta s_{t+1}) \le V\left(u-\sum_{t=0}^{T-1} x_t\cdot\Delta s_{t+1}\right) + V^*(y)\quad P\text{-a.s.}
\]
where for every $x\in\N^\infty$ and $y\in\dom EV^*$, the right side is integrable for some $u\in\U$. Thus, $-\sum_{t=0}^{T-1}x_t\cdot(y\Delta s_{t+1})$ is integrable for every $x\in\N^\infty$ and $y\in\dom EV^*$. Therefore the last infimum equals $-\infty$ unless $E_t(y\Delta s_{t+1})=0$ for every $t$, i.e.\ unless $y\in\Q$. Moreover, if $y\in\dom\tilde\varphi^*$ then $Ef^*(v,y)=\tilde\varphi^*(y)$ holds with $v_t=y\Delta s_{t+1}$ for $t<T$ and $v_T=0$, so the last condition of Theorem~\ref{thm:duality} holds.
\end{example}

The following example addresses a parameterized version of the generalized problem of Bolza studied in Rockafellar and Wets~\cite{rw83}.

\begin{example}\label{ex:bolza}
Consider the problem
\begin{equation}\label{bolza}
\minimize_{x\in\N}\quad E\sum_{t=0}^TK_t(x_t,\Delta x_t+u_t),
\end{equation}
where $n_t=d$, $\Delta x_t:=x_t-x_{t-1}$, $x_{-1}:=0$ and each $K_t$ is an $\F_t$-measurable normal integrand on $\reals^d\times\reals^d\times\Omega$. 

We assume that, for every $x\in\N^\infty$ with $x_t\in\cl\dom_1 K_t$ for all $t$, there exists $\bar x\in\N^\infty$ with $\bar x_t\in\rinterior \dom_1 K_t$ for all $t$ and $(\bar x,x)\in\dom_1 Ef$, and that $^a u\in\U$ and $^a y\in\Y$ for every $u\in\U$ and $y\in\Y$, where $^au_t=E_t u_t$. Then the closure of the value function $\varphi$ of \eqref{bolza} has the dual representation
\[
(\cl\varphi)(u) = \sup_{y\in\Y\cap \N} E\sum_{t=0}^T[u_t\cdot y_t-K_t^*(E_t\Delta y_{t+1},y_t)]
\]
for every adapted $u\in\U$. 

Indeed, the problem fits our general framework with
\begin{align*}
f(x,u,\omega) &= \sum_{t=0}^TK_t(x_t,\Delta x_t+u_t,\omega),\\
l(x,y,\omega) &= \sum_{t=0}^T\left[-x_t\cdot\Delta y_{t+1}+H_t(x_t,y_t,\omega)\right],\\
f^*(v,y,\omega)&= \sum_{t=0}^T K_t^*(v_t+\Delta y_{t+1},y_t),
\end{align*}
where $u=(u_0,\ldots,u_T)$ with $u_t\in\reals^d$ and
\[
H_t(x_t,y_t,\omega) := \inf_{u_t\in\R^d}\{K_t(x_t,u_t,\omega)-u_t\cdot y_t\}
\]
is the associated {\em Hamiltonian}.


The domain condition in Theorem~\ref{thm:duality} is satisfied, so
\[
-\tilde\varphi^*(y)=\inf_{x\in\N^\infty}E\underline l(x,y) =  \inf_{x\in\N^\infty} E\sum_{t=0}^T\left[-x_t\cdot\Delta y_{t+1}+\bunderline H_t(x_t,y_t)\right],
\]
where $\bunderline H_t(x_t,y_t,\omega)=\sup_{v_t}\{v_t\cdot x_t -K_t^*(v_t,y_t,\omega)\}$. Thus, by the interchange rule \cite[Theorem 14.50]{rw98}, 
\begin{align*}
-\tilde\varphi^*(y) &= -\inf_{x\in\N^\infty} E\sum_{t=0}^T\left[-x_t\cdot E_t\Delta y_{t+1}+\bunderline H_t(x_t,y_t,\omega)\right]\\
&= -E\sum_{t=0}^TK_t^*(E_t\Delta y_{t+1},y_t)
\end{align*}
for adapted $y$. Moreover, with $v_t:=E_t\Delta y_{t+1}-\Delta y_{t+1}$ we get $\tilde\varphi^*(y)=Ef^*(v,y)$, where $v\in\N^\perp$. The last condition in Theorem~\ref{thm:duality} thus holds. For any $y$, Jensen's inequality gives
\begin{align*}
-\tilde\varphi^*(y)&= \inf_{x\in\N^\infty}El(x,y)\\
&\le \inf_{x\in\N^\infty}E\sum_{t=0}^T[-x_t\cdot E_t\Delta y_{t+1} + H_t(x_t,E_ty_t)]\\
&=-\tilde\varphi^*({^a}y).
\end{align*}
Therefore, for adapted $u$, we get
\[
\cl\varphi(u)=\sup_{y\in\Y\cap \N} E\sum_{t=0}^T[u_t\cdot y_t-K_t^*(E_t\Delta y_{t+1},y_t)].
\]
\end{example}

The dual representation of the value function in Example~\ref{ex:bolza} was claimed to hold in \cite{pen11} under the assumption that the Hamiltonian is lsc in $x$. The claim is, however, false in general since it omitted the domain condition in Theorem~\ref{thm:duality}. The integrability condition posed in Example~\ref{ex:bolza} not only provides a sufficient condition for that, but it also makes the lower semicontinuity of the Hamiltonian a redundant assumption.

\section{Optimality conditions}\label{sec:oc}

The previous section as well as the articles \cite{pen11c,pp12} were concerned with dual representations of the value function $\varphi$. Continuing in the general conjugate duality framework of Rockafellar~\cite{roc74}, this section derives optimality conditions for \eqref{p} by assuming the existence of a subgradient of $\varphi$ at $u$. Besides optimality conditions, this assumption implies the lower semicontinuity of $\varphi$ at $u$ (with respect to the weak topology induced on $\U$ by $\Y$) and thus, the absence of a duality gap as well. Whereas in the above reference, the topology of convergence in measure in $\N$ played an important role, below, topologies on $\N$ are irrelevant.

Recall that a $y\in\Y$ is a {\em subgradient} of $\varphi$ at $u\in\U$ if
\[
\varphi(u')\ge\varphi(u) + \langle u'-u,y\rangle\quad\forall u'\in\U.
\]
The set of all such $y$ is called the {\em subdifferential} of $\varphi$ at $u$ is and it is denoted by $\partial\varphi(u)$. If $\partial\varphi(u)\ne\emptyset$, then $\varphi$ is lower semicontinuous at $u$ and
\[
\varphi(u)=\langle u,y\rangle -\varphi^*(y)
\]
for every $y\in\partial\varphi(u)$. By \cite[Theorem~11]{roc74}, $\partial\varphi(u)\ne\emptyset$, in particular, when $\varphi$ is continuous at $u$.

We assume from now on that $Ef$ is closed in $u$ and that $\varphi$ is proper.

\begin{theorem}\label{thm:sp}
Assume that $\partial\varphi(u)\ne\emptyset$ and that for every $y\in\partial\varphi(u)$ there exists $v\in\N^\perp$ such that $\varphi^*(y)=Ef^*(v,y)$. Then an $x\in\N$ solves \eqref{p} if and only if it is feasible and there exist $y\in\Y$ and $v\in\N^\perp$ such that 
\[
(v,y)\in\partial f(x,u)
\]
$P$-almost surely, or equivalently, if
\begin{align*}
v\in\partial_x l(x,y)\quad\text{and}\quad u\in\partial_y[-l](x,y)
\end{align*}
$P$-almost surely.
\end{theorem}

\begin{proof}
Note first that if $y\in\partial\varphi(u)$ and $v\in\N^\perp$ such that $\varphi^*(y)=Ef^*(v,y)$, then $\varphi(u)+\varphi^*(y)=\langle u,y\rangle$ and
\[
\inf_{x\in\N} Ef(x,u)=\langle u,y\rangle-\varphi^*(y)=E[u\cdot y -f^*(y,v)].
\]
Thus, $x$ solves \eqref{p} if and only if
\[
E[f(x,u)+f^*(v,y)]= E[u\cdot y].
\]
By Fenchel's inequality,
\[
f(x,u)+f^*(v,y)\ge x\cdot v + u\cdot y,
\]
so, by Lemma~\ref{lem:1}, $E[x\cdot v]=0$ for every feasible $x$. Thus, $x$ solves \eqref{p} if and only if $x$ is feasible and the above inequality holds as an equality $P$-almost surely, that is, if
\[
(v,y)\in\partial f(x,u)
\]
$P$-almost surely. By \cite[Theorem 37.5]{roc70a}, this is equivalent to  
\begin{align*}
v\in\partial_x l(x,y)\quad\text{and}\quad u\in\partial_y[-l](x,y)
\end{align*}
$P$-almost surely.


\end{proof}

\begin{example}\label{ex:ie2}
In Example~\ref{ex:ie}, the optimality conditions of Theorem~\ref{thm:sp} mean that
\begin{align*}
f_j(x) + u_j &\le 0,\\
x \in\argmin_{z\in\reals^n}\{f_0(z)&+\sum_{j=1}^my_jf_j(z) - z\cdot v\},\\
y_jf_j(z) &=0\quad j=1,\ldots,m,\\
y_j&\ge 0
\end{align*}
$P$-almost surely. These are the optimality conditions derived in \cite{rw78}, where sufficient conditions were given for the existence of an optimal $x\in\N^\infty$ and the corresponding dual variables $y\in\Y$ and $v\in\N^\perp$ (in our notation); see \cite[Theorem~1]{rw78}. The conditions of \cite{rw78} imply the continuity of the optimum value at the origin with respect to the $L^\infty$-norm. This yields the existence of dual variables in the norm dual $(L^\infty)^*$. They then used the condition of ``relatively complete recourse'' to show that the projections of the dual variables to the subspace $L^1\subset(L^\infty)^*$ satisfy the optimality conditions as well.
\end{example}

We will now describe another general setup which covers many interesting applications and where the subdifferentiability condition $\partial\varphi(u)\ne\emptyset$ in Theorem~\ref{thm:sp} is satisfied. This framework is motivated by Biagini~\cite{bia8}, where similar arguments were applied to optimal investment in the continuous-time setting. The idea is simply to require stronger continuity properties on $\varphi$ in order to get the existence of dual variables directly in $\Y$ (without going through the more exotic space $(L^\infty)^*$ first as in \cite{rw78}).

A topological vector space is said to be {\em barreled} if every closed convex absorbing set is a neighborhood of the origin. By \cite[Corollary~8B]{roc74}, a lower semicontinuous convex function on a barreled space is continuous throughout the algebraic interior (core) of its domain. On the other hand, by \cite[Theorem~11]{roc74}, continuity implies subdifferentiability. Fr\'echet spaces and, in particular, Banach spaces are barreled. In the following, we say that $\U$ is {\em barreled} if it is barreled with respect to a topology compatible with the pairing with $\Y$.

The following applies Theorem~\ref{thm:sp} to optimal investment in perfectly liquid financial markets.

\begin{example}\label{ex:optinv2}
Consider Example~\ref{ex:optinv} and assume that $\U$ is barreled, $EV$ is finite on $\U$ and $EV^*$ is proper in $\Y$. 
Then an $x\in\N$ solves \eqref{alm} if and only if it is feasible and there exists $y\in\Q$ such that 
\begin{align*}
y&\in\partial V(u-\sum_{t=0}^{T-1}x_t\cdot\Delta s_{t+1})\quad P\text{-a.s.}
\end{align*}
\end{example}

\begin{proof}
By \cite[Theorem~21]{roc74}, $EV$ is lower semicontinuous so by \cite[Corollary~8B]{roc74}, it is continuous. Since
\[
\varphi(u) \le Ef(0,u) = EV(u),
\]
\cite[Theorem~8]{roc74} implies that $\varphi$ is continuous and thus subdifferentiable throughout $\U$, by \cite[Theorem~11]{roc74}.
Moreover, $\varphi^*(y)=Ef^*(v,y)$ for every $y\in\dom\varphi^*$ and $v\in\N^\perp$ given by $v_t=y\Delta s_{t+1}$. The assumptions of Theorem~\ref{thm:sp} are thus satisfied. The subdifferential conditions for the Lagrangian integrand $l$ can now be written as
\begin{align*}
v_t &= -y\Delta s_{t+1}\quad P\text{-a.s. for $t< T$ and $v_T=0$},\\
u-\sum_{t=0}^{T-1}x_t\cdot\Delta s_{t+1} &\in \partial V^*(y)\quad P\text{-a.s.}
\end{align*}
Since $v\in\N^\perp$, the former means that $y\in\Q$ while the latter can be written in terms of $\partial V$ as stated; see \cite[Theorem~12]{roc74}.
\end{proof}

The optimality conditions in Example~\ref{ex:optinv2} are classical in financial mathematics; for continuous-time models, see e.g.\ Schachermayer~\cite{sch1} or Biagini and Frittelli~\cite{bf8} and the references therein.

The next example applies Theorem~\ref{thm:sp} to the problem of Bolza from Example~\ref{ex:bolza}. It constructs a subgradient $y\in\partial\varphi(u)$ using Jensen's inequality. 


\begin{example}\label{ex:bolza2}
Consider Example~\ref{ex:bolza} and assume that $\U$ is barreled and that there exists a normal integrand $\theta$ such that $E\theta$ is finite on $\U$, $E\theta^*$ is proper on $\Y$, and, for every $u\in\U\cap\N$,
\[
\sum_{t=0}^T K_t(x_t,\Delta x_t+u_t)\le\theta(u) \quad P\text{-a.s.}
\]
for some $x\in\N$. Then an $x\in\N$ solves 
\[
\minimize_{x\in\N}\quad E\sum_{t=0}^TK_t(x_t,\Delta x_t+u_t)
\]
for a $u\in\U\cap\N$ if and only if $x$ is feasible and there exists $y\in \Y\cap\N$ such that 
\[
(E_t\Delta y_{t+1},y_t)\in\partial K_t(x_t,\Delta x_t+u_t)
\]
$P$-almost surely for all $t$, or equivalently, if
\begin{align*}
E_t\Delta y_{t+1} &\in \partial_{x} H_t(x_t,y_t),\\
u_t+\Delta x_t &\in \partial_{y}[-H_t](x_t,y_t)
\end{align*}
$P$-almost surely for all $t$.
\end{example}

\begin{proof}
The space $\U\cap\N$ is a barreled space and its continuous dual may be identified with $\Y\cap\N$. Indeed, by Hahn--Banach, any $y\in(\U\cap\N)^*$ may be extended to a $\bar y\in\Y$ for which $^a\bar y\in\Y\cap\N$ coincides with $y$ on $\U\cap\N$. Moreover, for any closed convex absorbing set $B$ in $\U\cap\N$, we have that $\hat B=\{u\in\U\mid {^a u}\in B \}$ is a closed convex absorbing set in $\U$, so it is a neighborhood of the origin of $\U$ which implies that $\hat B\cap\N\subset B$ is a neighborhood of the origin of $\U\cap\N$.

By assumption, for every $u\in\U\cap \N$ there is an $x\in\N$ such that
\begin{align*}
\varphi(u) &\le  E\sum_{t=0}^T K_t(x_t,\Delta x_t+u_t) \le E\theta(u).
\end{align*}
By \cite[Theorem~21]{roc74}, $E\theta$ is lower semicontinuous so by \cite[Corollary~8B]{roc74}, it is continuous. Thus, $\varphi$ is continuous and in particular, subdifferentiable on $\U\cap\N$ (see \cite[Theorem~11]{roc74}) so for every $u\in\U\cap\N$ there is a $y\in \Y\cap\N$ such that 
\[
\varphi(u')\ge \varphi(u)+\langle u'-u,y\rangle\quad\forall u'\in\U\cap\N. 
\]
Since each $K_t$ is $\F_t$-measurable, Jensen's inequality gives
\[
\varphi(u')\ge\varphi({^a u'})\ge \varphi(u)+\langle ^au'-u, y\rangle=\varphi(u)+\langle u'-u, y\rangle\quad\forall\ u'\in\U,
\]
so $y\in\partial\varphi(u)$ as well. Moreover, as observed in Example~\ref{ex:bolza}, we have $Ef^*(v,y)=\varphi^*(y)$ for $v_t= E_t\Delta y_{t+1}-\Delta y_{t+1}$. The subdifferential conditions in Theorem~\ref{thm:sp} for the Lagrangian integrand $l$ become
\begin{align*}
v_t +\Delta y_{t+1}&\in \partial_{x} H_t(x_t,y_t)\ P\text{-a.s.}\quad\forall t,\\
u_t+\Delta x_t &\in \partial_{y}[-H_t](x_t,y_t)\ P\text{-a.s.}\quad\forall t.
\end{align*}
By \cite[Theorem 37.5]{roc70a}, these are equivalent to the conditions in terms of $K_t$.
\end{proof}

The optimality condition in terms of $K_t$ in Example~\ref{ex:bolza2} can be viewed as a {\em stochastic Euler-Lagrange condition} in discrete-time much like that in \cite[Theorem~4]{rw83}. There is a difference, however, in that \cite{rw83} studied the problem of minimizing $E\sum K_t(x_{t-1},\Delta x_t)$ and, accordingly, the measurability conditions posed on the dual variables were different as well. The condition in terms of $H_t$ can be viewed as a {\em stochastic Hamiltonian system} in discrete-time; see \cite[Section~9]{roc70b} for deterministic models in continuous-time. Our assumptions on the problem also differ from those made in \cite{rw83}. Whereas the assumptions of \cite{rw83} and the line of argument follows that in \cite{rw78} (see Example~\ref{ex:ie2}), our assumption implies the continuity of $\varphi$ on the adapted subspace $\U\cap\N$.

It is essential that the growth condition in Example~\ref{ex:bolza2} is required only for adapted $u\in\U£$. Indeed, since $x$ is adapted, it would often be too much to ask the upper bound for nonadapted $u$. This is the case e.g.\ in the following example which is concerned with the financial model studied in \cite{pp10,pp12}. The model is an extension of the currency market model introduced by Kabanov~\cite{kab99}; see also Kabanov and Safarian~\cite{ks9}.

\begin{example}\label{ex:kabanov}
Consider the optimal investment-consumption problem
\begin{equation}\label{ocp}\tag{OCP}
\begin{aligned}
&\minimize & &E\sum_{t=0}^T V_t(-k_t)\quad\ovr\quad (z,k)\in\N\\
&\st  &  \Delta z_t&+u_t + k_t \in C_t,\ z_T=0\quad P\text{-a.s.}\ t=0,\ldots,T,
\end{aligned}
\end{equation}
where $z_{-1}:=0$, and for each $t$, $C_t:\Omega\tos\reals^d$ is $\F_t$-measurable\footnote{$C_t$ is {\em $\F_t$-measurable} if  $\{\omega\in\Omega\mid C_t(\omega)\cap O\ne\emptyset\}\in\F_t$ for every open $O$.} and closed convex-valued with $0\in C_t$ and $V_t$ is an $\F_t$-measurable nondecreasing (wrt $\reals^d_+$) convex normal integrand with $V_t(0)=0$.

We will assume as in Example~\ref{ex:bolza} that adapted projections of the elements of $\U$ and $\Y$ stil belong to $\U$ and $\Y$, respectively. We also assume that $\U$ is barreled and that there exists a normal integrand $\theta$ such that $E\theta$ is finite on $\U$, $E\theta^*$ is proper on $\Y$, and that
\[
V_T\left(\sum_{t=0}^T u_t\right)\le\theta(u)\quad\forall u\in\reals^{(T+1)d}.
\]
Then $(z,k)\in\N$ solves \eqref{ocp} for $u\in \U\cap\N$ if and only if $(z,k)$ is feasible and there exists a martingale $y\in\Y$ such that\footnote{Here, $\sigma_{C_t}(y_t):=\sup\{z\cdot y_t\,|\,z\in C_t\}$.}
\begin{align*}
-k_t &\in\partial V^*_t(y_t),\\
\Delta z_t+u_t+k_t &\in \partial \sigma_{C_t}(y_t)
\end{align*}
$P$-almost surely for $t=0,\ldots,T$. When $C_t$ are conical, the last condition means that
\[
\Delta z_t+u_t+k_t\in C_t,\ y_t\in C_t^*,\ (\Delta z_t+u_t+k_t)\cdot y_t = 0
\]
$P$-almost surely for $t=0,\ldots,T$. Thus, $y$ generalizes the notion of a ``consistent price system'' introduced in \cite{kab99} to nonconical market models; see \cite{pp12}.
\end{example}

\begin{proof}
Problem \eqref{ocp} can be expressed in the format of Example~\ref{ex:bolza2} with $x=(z,k)$, $u=(u^z,u^k)$ and
\begin{align*}
K_t(x_t,u_t,\omega) &=V_t(-k_t,\omega)+\delta_{C_t(\omega)}(u^z_t+k_t)\quad t=0,\dots,T-1,\\
K_T(x_T,u_T,\omega) &=V_T(-k_T,\omega)+\delta_{C_T(\omega)}(u^z_T+k_T)+\delta_{\{0\}}(z_T).
\end{align*}
The Hamiltonian becomes
\begin{align*}
H_t(x_t,y_t,\omega) &= \delta_{\{0\}} (y^k_t)+k_t\cdot y^z_t+V_t(-k_t,\omega)-\sigma_{C_t(\omega)}(y^z_t)\quad t=0,\dots,T-1,\\
H_T(x_T,y_T,\omega) &= \delta_{\{0\}} (y^k_t)+k_T\cdot y^z_T+V_T(-k_T,\omega)-\sigma_{C_T(\omega)}(y_T)+\delta_{\{0\}}(z_T).
\end{align*}
We are in the setting of Example~\ref{ex:bolza}, since the domain condition is satisfied by $\bar x=(0,\bar k)$ for any $x\in\N^\infty$, where $\bar k_t=1$ for all $t$. 

Given $u\in \U\cap\N$, we have
\[
\sum_{t=0}^T K_t(x^u_t,\Delta x_t+u_t,\omega) = V_T\left(\sum_{t=0}^T u_t\right) \le\theta(u),
\]
where $x^u=(z^u,k^u)\in\N$ is defined by $\Delta z^u_t=-u_t$, $k^u_t=0$ for $t<T$ and $z^u_T=0$ and $k^u_T=-\sum_{t=0}^Tu_t$. Thus, the assumptions of Example~\ref{ex:bolza2} are satisfied with $y=0$ and $\beta=0$. The given optimality conditions are just a special case of the optimality conditions in Example~\ref{ex:bolza2}. Indeed, the subdifferential conditions become
\begin{align*}
z_T= 0 \text{ and } E_t\Delta y^z_{t+1} &= 0 \quad\forall\ t=0,\dots,T-1,\\
E_t\Delta y^k_{t+1} &\in y^z_t - \partial V_t(-k_t)\quad\forall\ t=0,\dots,T,\\
u^z_t+\Delta z_t &\in -k_t+\partial\sigma_{C_t}(y^z_t)\quad\forall\ t=0,\dots,T,\\
y^k_t &=0\quad\forall\ t=0,\dots, T.
\end{align*}
The first condition means that $y^z$ is a martingale and $z_T=0$. The second and the last condition mean that $y^k=0$ and $y^z_t\in\partial V_t(-k_t)$. By \cite[Theorem~12]{roc74}, the subdifferential condition can written as $-k_t \in\partial V^*_t(y_t)$. 
\end{proof}

\bibliographystyle{plain}
\bibliography{sp}

\end{document}